\documentclass{amsart}

\newtheorem{theorem}{Theorem}[section]
\newtheorem{lemma}[theorem]{Lemma}
\newtheorem{corollary}[theorem]{Corollary}
\theoremstyle{definition}
\newtheorem{definition}[theorem]{Definition}

\theoremstyle{remark}
\newtheorem{remark}[theorem]{Remark}

\numberwithin{equation}{section}

\newcommand{\R}{\mathbb{R}}
\newcommand{\C}{\mathbb{C}}
\newcommand{\hC}{\widehat{\mathbb{C}}}

\newcommand{\jl}{\mathcal{J}(f)}

\begin{document}

\title{Herman rings of meromorphic maps with an omitted value}

\author{Tarakanta Nayak}
\address{School of Basic Sciences, Indian Institute of Technology Bhubaneswar, India}
\email{tnayak@iitbbs.ac.in}
\thanks{The author is supported by the Department of Science \& Technology, Govt. of India through the
 Fast Track Project (SR/FTP/MS-019/2011).}



\date{Dec 30, 2014.}



\begin{abstract}
We investigate the existence and distribution of Herman rings of transcendental
meromorphic
 functions which have at least one omitted value. If all the poles of such a function are multiple then
   it has no Herman ring. Herman rings of period one or two do not exist.
Functions with a single pole or with at least two poles one of which is an omitted value have no Herman ring. Every doubly connected
 periodic Fatou component is a Herman ring.
\end{abstract}

\maketitle

%

\section{Introduction}

Unlike rational maps, a transcendental meromorphic map can omit a
point of the Riemann sphere. Such a point is called an omitted value
of the map and  by Picard's theorem, there can be at most two such
values. These values are known to be asymptotic values i.e., those
to which the map $f(z)$ approaches when $z \to \infty$ along some
curve. Not every asymptotic value is omitted.
 \par
 A singular value
  of a transcendental meromorphic map $f$ is either a critical value (the image of a point $z$  for which $f'(z)=0$) or an asymptotic value.
A singular value or a limit point of singular values is known as a singularity of the inverse function $f^{-1}$  because this is a point where at least one branch of $f^{-1}$ fails
 to be defined. Further, there are different possible ways
   in which this failure can take place leading to the following classification of singularities \cite{bergere95}.
   For $a \in \widehat{\C}$ and $r>0$, let $D_r(a)$ be a disk (in the spherical
  metric) and
choose a component $U_r$ of $f^{-1}(D_r(a))$ in such a way that
$U_{r_1}\subset U_{r_2}$ for $0<r_1<r_2$.
 There are two possibilities.
\begin{enumerate}
\item   $\bigcap_{r>0} U_r= \{z\}$ for $z \in \C$:  In this case, $f(z)=a$.
The point $z$ is called an \textsl{ordinary point} if (i)$f'(z)\neq
0$ and $a \in \C$, or (ii)  $z$ is a simple pole. The point $z$ is
called a \textsl{critical point} if $f'(z)=0$  and $a \in \C$, or
$z$ is a
 multiple pole. In this case, $a$ is called a \textit{critical value} and we say that a critical
 point/algebraic singularity lies over $a$.
 \item $\bigcap_{r>0} U_r= \emptyset$: The choice $r \to U_r$ defines
  a   transcendental  singularity of $f^{-1}$.
We say a singularity $U$ lies over $a$. The singularity $U$ lying over $a$
is called direct if there exists $r>0$ such that $f(z) \neq a$ for
all $z \in U_r$. Otherwise it is called indirect.
\end{enumerate}

 Over each asymptotic value, there lies a transcendental singularity and there is always a critical point lying over a critical
   value. It is important to note that an asymptotic value can also
   be a critical value. But an omitted value can neither be a critical
   value nor the image of any ordinary point. Further, each singularity lying over an omitted value is
   direct. In this way, an omitted
  value can be viewed as the simplest instance of a transcendental singularity.
\par
A transcendental meromorphic function  (for which $\infty$ is the only essential singularity) can (1) be entire, (2) be
analytic self-map of the punctured plane i.e., with only one pole
which is an omitted value or (3) have at least two poles or
  exactly one pole which is not an omitted value.  The functions in the last category are usually
   referred as  general meromorphic functions  possibly because the property of being meromorphic has the clearest manifestation,
   at least in dynamical terms in this case.
   Let $M$ denote the class of all general meromorphic maps and $$ M_o =\{f \in M~:~ f ~\mbox{has at least an omitted value} \}.$$
We deal with the functions belonging to the class $M_o$ in this article.
   \par
 For a meromorphic function $f:
\mathbb{C} \rightarrow \widehat{\mathbb{C}}=\mathbb{C} \cup \{
\infty \}$, the set of points $z \in \widehat{\mathbb{C}}$ in a
neighbourhood of which the sequence of iterates
$\{f^n\}_{n=0}^{\infty}$ is defined and forms a normal family is
called the Fatou set of $f$. The Julia set is its complement in
$\widehat{\mathbb{C}}$. The Fatou set is open by definition and each of its
maximal connected subset is known as a Fatou component. A
 Fatou component $U$ is called $p$-periodic if $p$ is the smallest natural number satisfying $f^p(U) \subseteq U$.
 Periodic Fatou components
are of five types, namely Attracting domain, Parabolic domain,
Siegel disk, Herman ring and Baker domain. A Herman ring $H$ with period $p$ is such that there exists
an analytic homeomorphism $\phi:H \to A =\{z:1<|z|<r\}, r>1$ with
$\phi(f^p(\phi^{-1}(z)))=e^{i 2\pi \alpha}z$ for all $z \in A$ and for some $\alpha \in
\R\setminus \mathbb{Q}$. Clearly, there are uncountably many $f^p$-invariant Jordan curves in $H$. Each such curve separates the two components of $\widehat{\C} \setminus H$. By ring, we shall mean Herman ring throughout this article.
   \par
    Transcendental entire functions always
   omit $\infty$ and at most another point in the plane. Analytic
   self-maps of the punctured plane with only one essential singularity omit only one finite value, namely the pole. It is well-known that entire functions and
analytic self-maps of the punctured plane cannot have any Herman
   ring, the proof of the later appearing in ~\cite{gong-1996}. Investigations on the role of omitted values in determining certain aspects of the dynamics
of a function is initiated in~\cite{tkzheng-omitted}.
It is seen, among other things, that in most of the cases a multiply connected Fatou component of a function with at least one omitted value ultimately lands on a Herman ring of period at least $2$.  The current article investigates the existence of the Herman rings of meromorphic functions with at least an omitted value.
\par
There are many different possible arrangements of a $p-$periodic cycle of Herman rings in the plane.
To understand this we make the following definition. For a Herman ring $H$, we denote the Herman ring containing $f^i(H)$ by $H_i$ for each non-negative natural number $i$ throughout this article. Here $H_0$ stands for $H$. Let $B(H)$
denote the bounded component of $\widehat{\mathbb{C}} \setminus H$.
\begin{definition} (\textbf{$H-$maximal nest})\\
Given a Herman ring $H$, a ring $H_j$ is called an $H-$outermost ring if $H_j$ is not contained in $B(H_i)$ for any $i, i \neq j$. Given an $H-$outermost ring $H_j$, the collection of rings consisting of $H_j$ and all $H_i$
such that $H_i \subset B(H_j)$ is called an $H-$maximal nest.
\end{definition}

Note that an $H-$maximal nest is a sub-collection of Herman rings from the periodic cycle containing $H$.
  The number of $H-$maximal nests can be any natural number less than or equal to the period of $H$ and each $H-$maximal nest corresponds to an $H-$outermost ring. One way to broadly classify the possible arrangements of a $p-$periodic cycle of Herman rings may be in terms of the number of maximal nests.
The two extreme arrangements, namely when the number of maximal nests is $1$ or $p$, deserve names of their own.
\begin{definition}(\textbf{Nested, Strictly nested and Strictly non-nested})
A $p-$periodic cycle of Herman rings with $p >1$ is called \textit{nested} if there is a $j$ such that $H_i \subset B(H_j)$ for all $i \neq j$. It is called \textit{strictly nested} if for each $i \neq j$, either $ H_i \subset B(H_j)$ or $H_j \subset B(H_i)$. We say $H$ is \textit{strictly non-nested} if $B(H_i) \bigcap B(H_j) = \emptyset$ for all $i \neq j$. Every Herman ring of a nested or strictly nested cycle is also called nested or strictly non-nested respectively.
\end{definition}
It is clear that a Herman ring $H$ is nested if there is only one $H-$maximal nest. Each Herman ring of period two is either nested or strictly non-nested.
\par
It is shown in~\cite{tkzheng-omitted} that functions belonging to $M_o$ having only one pole have no Herman ring of period $2$. That this is true for functions possibly with more than one pole is one of the implications of the following result.

\begin{theorem}
If $f \in M_o$ then $f$ has no Herman ring which is nested or
strictly non-nested and in particular, it has no Herman ring of
period one or two. Further, if a pole of $f$ is an omitted value then it
has no Herman ring of any period. \label{nested-nonnested HR}
\end{theorem}
Theorem~\ref{nested-nonnested HR} gives that if a function $f \in M_o$ has a pole which is also an omitted value then it cannot have any Herman ring. Here we look at the possibility of Herman rings for functions with only a single pole which is not necessarily an omitted value.  An easy-to-verify sufficient condition for non-existence of Herman rings follows.
\begin{theorem}
If $f \in M_o$  has only one  pole then $f$ has no Herman ring.
\label{onepole-no HR}\end{theorem}
For each entire map $g$ and non-zero complex number $z_0$, the map
$f(z)= \frac{e^{g(z)}}{(z-z_0)^k}$ is meromorphic with only a single
pole $z_0$ which is different from its omitted value where $k$ is a natural number.
 By the above theorem, it has no Herman ring.
\par
To deal with Herman rings $H$ with exactly two $H-$maximal nests, we first make a definition.
\begin{definition}(\textbf{Strictly nested $H-$maximal nest})
An $H-$maximal nest $N= \{H_{i_1},~H_{i_2},~H_{i_3,...,~H_{i_n}}\}$ is called strictly nested if for all
$1 \leq j \neq k  \leq n$, either $H_{i_j} \subset B(H_{i_k})$ or $H_{i_k} \subset B(H_{i_j})$.
\end{definition}
Every strictly nested $H-$maximal nest has a unique innermost ring.
We present a result showing how a particular arrangement of rings can force the period of the ring to be odd.
\begin{theorem} For a Herman ring $H$ of a function belonging to $M_o$, if there are only two $H$-maximal nests, one of
which consists of only one ring and the other is strictly nested,
then $H$ is odd periodic.
\label{oddperiodic}\end{theorem}
For a $3-$periodic Herman ring $H$, it follows from Theorem~\ref{nested-nonnested HR} that there are two $H-$maximal nests. It is obvious that one of these is strictly nested and the other consists of only one ring. Thus the converse of the above result is true for $3-$periodic Herman rings.
\par
Herman ring is a periodic doubly connected Fatou component. But
  the converse is not at all obvious.
It is an open question that whether a doubly connected periodic
Fatou component of a meromorphic function (with a single essential singularity) is always a Herman ring~\cite{bolsch99}. This is known to be true when period is one~\cite{bk4}.
We settle this question for all periods and for all maps belonging to $M_o$.
\begin{theorem} For every $f \in M_o$, each doubly connected periodic Fatou component is a Herman ring.
\label{open}
\end{theorem}
 Using  quasi-conformal maps,  Fagella et al.  investigated
   Herman rings of transcendental maps and proved that a general meromorphic function having $n$ poles cannot have more than $n$ invariant Herman rings~\cite{fagella-peter 2012}.
   The same tool is exploited by Zheng for showing that a function of finite type has at most finitely many Herman rings~\cite{zheng-2000}. Zheng's arguments are very different from the ones used in ~\cite{fagella-peter 2012}. He uses quasi-conformal deformation of the function which is similar to Sullivan's proof of the absence of wandering domains for rational functions. Though we establish similar restrictions on Herman rings for functions belonging to $M_o$, our approach does not use quasi-conformal maps and is mostly elementary.
   \par
In Section 2, a number of lemmas are proved that are required for the proofs later.
Also an easy corollary establishing non-existence of Herman rings whenever all poles are multiple is proved in this section.  We provide the proofs of all the results in Section 3.
\par
 For a closed curve $\gamma$ in $\C$, let $B(\gamma)$ denote
the union of all the bounded components of $\widehat{\C} \setminus
\gamma$. For a doubly connected domain $H$, $B(H)$ means the
bounded component of $\widehat{\C} \setminus H$ and $\widetilde{H}$
denotes $H \bigcup B(H)$. We say a closed curve (or a Herman ring $H$) surrounds a point $a$ if $a \in B(\gamma)$ (or $a \in B(H)$). The boundary and the closure of a domain
$D$ in $\hC$ is denoted by $\partial D$ and $\overline{D}$
respectively. A maximal connected subset of the Julia set is called a Julia component.
Denote the component of the Julia set $\jl$ containing a set $A$ by
   $J_A$. Also $O_f$ stands for the set of all omitted values of $f$. For a Fatou component $V$ of $f$, we denote the Fatou component containing $f^n(V)$ by $V_n$ for $n=0,1,2,...$ where $f^0$ denotes the identity map.
   \section{Preliminary lemmas}
A non-contractible Jordan curve in the Fatou set of $f \in M$ eventually
(under forward iteration)
   surrounds a pole and in the next iteration, all the omitted values whenever such values exist.
This elementary but useful fact is already proved in
   \cite{tkzheng-omitted}. But we state and prove it here for completeness.
\begin{lemma} Let $f \in M$ and $V$ be a multiply
connected  Fatou component of $f$. Suppose that
$\gamma$ is a closed curve in $V$ with $B(\gamma) \bigcap
\mathcal{J}(f) \neq \emptyset$. Then there is an $n \in \mathbb{N}
\bigcup \{0\}$ and a closed curve $\gamma_n \subseteq f^n(\gamma)$
in $V_n$ such that $B(\gamma_n)$ contains a pole of $f$.
  Further, if $O_f \neq \emptyset$ then $O_f \subset B(\gamma_{n+1})$
  for some closed curve $\gamma_{n+1}$ contained in $f(\gamma_n)$.
\label{general1}
\end{lemma}
\begin{proof}
 Since $f \in M$ and $B(\gamma)\bigcap  \mathcal{J}(f) \neq
\emptyset$, there exists a $z \in B(\gamma)$ satisfying
$f^k(z)=\infty$ for some $k \in \mathbb{N}$. The set $\mathcal{N}=\{
m \in \mathbb{N}~:~ f^{m}(z)=\infty~\mbox{for some }~ z \in
B(\gamma) \}$ is a non-empty subset of $\mathbb{N}$ and  the Well-Ordering Property
 of natural numbers gives that $\mathcal{N}$ has a smallest element. Let it be
$\tilde{n}$ and set $n=\tilde{n}-1$. Then $n \in  \mathbb{N} \bigcup
\{0\}$ and $f^{n}~:~B(\gamma)   \to
 \mathbb{C}$   is analytic.   Hence,  $\gamma_n = \partial (f^{n}(B(\gamma)))$
 is a  closed curve contained in $V_n$ with
$\gamma_n \subseteq f^{n}(\gamma)$ and $B(\gamma_{n})$ contains a
pole of $f$.
\par
Suppose that the closure of $f(B(\gamma_n))$ contains an  element
$a$ of $O_f$. Let $\{w_k\}_{k>0}$ be a sequence in $f(B(\gamma_n))$
converging to $a$ and for each $k$, let $z_k$ be a point  in
$B(\gamma_n)$ satisfying $f(z_k)=w_k$. Then, considering a limit
point $z_0$ of $\{z_{k}\}_{k>0}$ we observe that $z_0 \in
\overline{B(\gamma_n)}$. The continuity of $f$ at $z_0$ gives that
$f(z_0)=a$. This is a contradiction since $a$ is an omitted value.
Therefore, $O_f \subset \widehat{\mathbb{C}} \setminus
\overline{f(B(\gamma_n))}$. The set $\overline{f(B(\gamma_n))}$ is
connected and contains a neighborhood of $\infty$. Consequently, $
\widehat{\mathbb{C}} \setminus \overline{f(B(\gamma_n))}$ is a
non-empty open set whose boundary is contained in $f(\gamma_{n})$
and there is a closed (and bounded but not necessarily simple) curve
$\gamma_{n+1} \subseteq f(\gamma_n)$ such that
 $O_f \subset B(\gamma_{n+1})$.
\end{proof}
\begin{remark}
The  proof of the above lemma also gives that $O_f \bigcap
\overline{f(B)} =\emptyset$ for every bounded domain $B$.
\label{sb pole enclosed}
\end{remark}
 If $H$ is a Herman ring of $f \in M_o$ then some $H_j$ surrounds a pole and its forward image $H_{j+1}$
 surrounds $O_f$. Consequently, the Julia component containing such a pole or an omitted value is always bounded. Some other properties are necessary for a pole so that it can be surrounded by some Herman ring of functions belonging to $M_o$.
\begin{definition}(\textbf{Single separated})\\
A simple pole of a function is called  \textit{single separated} if the component of the Julia set containing it,
is bounded and does not contain any other pole of the function.
\end{definition}

\begin{lemma} If  $H$ is a Herman ring of $f \in M_o$ then $f: B(H) \to \widehat{\C}$ is one-one.
\label{one-one}
\end{lemma}
\begin{proof}
Let $H$ be a $p$-periodic Herman ring and $\gamma$ be an
$f^p$-invariant Jordan curve in it. Since $f: H \to \C$ is one-one
by the definition of Herman ring, $f(\gamma)$ is a Jordan curve winding around every point of $B(H_1)$ exactly once.
 \par Assume $f: B(H) \to
\widehat{\C}$ is not one-one which implies that $f: B(\gamma) \to \widehat{\C}$ is
not one-one.
\par
Let $f: B(\gamma) \to \widehat{\C}$ be analytic. Then there are at
least two points in $B(H)$ with the same image and by the Argument
principle, the curve $f(\gamma)$ winds around every point of $B(H_1)$ at
least twice. If $f: B(\gamma) \to \widehat{\C}$ has at least a pole
then Lemma~\ref{general1} ensures that $f(\gamma)$ surrounds an
omitted value $a$ of $f$. It follows from the assumption that $f$ has at least two poles or one pole with multiplicity at least two in $B(\gamma)$. Applying the Argument principle to $g(z)=
f(z)-a$ on $B(\gamma)$, we find $g(\gamma)$ winding around $0$ at
least twice implying that $f(\gamma)$  winds around $a$  at least
twice.  This is a contradiction which completes the proof.
\end{proof}
\begin{remark}
The lemma above in fact gives that if $f \in M_o$ then every pole surrounded by a Herman ring
is single separated and a Herman ring can surround at most one pole. Together with Lemma~\ref{general1}, it gives the following as a corollary.
\label{HR encloses a sb pole}\end{remark}
\begin{corollary}
If all the poles of a function belonging to $M_o$ are multiple then it has no Herman ring.
\end{corollary}

Now, the number of rings in a periodic cycle of Herman rings that
 surround some pole is to be determined. Recall that, for a Herman ring $H$,
  the Fatou component containing $f^i (H)$ is denoted by $H_i$.
\begin{lemma}
For every $p$-periodic Herman ring $H$ of $f \in M_o$, the number of elements in the set $\{0 \leq i \leq p~:~H_i ~\mbox{surrounds a pole}\}$  is even.
 \label{even}\end{lemma}
\begin{proof}
Let $H_j$ be a $p$-periodic Herman ring
and $\gamma_1, \gamma_2$ be two $f^p$-invariant Jordan curves in
$H_j$ such that $\gamma_1 \subset B(\gamma_2)$. If $H_j$ surrounds a pole then $f(\gamma_2) \subset B(f(\gamma_1))$. Otherwise
$f(\gamma_1) \subset B(f(\gamma_2))$. This is true for every $j \geq 0$ by Lemma~\ref{general1} and ~\ref{one-one}. If the number of
Herman rings in this cycle that surrounds some pole is odd then
$f^p(\gamma_1)$ surrounds $f^{p}(\gamma_2)$ which is a contradiction
as $f^p(\gamma_i)=\gamma_i$ for $i =1,2$.
\end{proof}
Given a Herman ring $H$ of period $p>1$, let $\gamma$ be an
$f^p$-invariant Jordan curve in it and set $\gamma_j=f^j (\gamma)$
for $j=0,1,2,...,p-1$.
 Each $\gamma_j$ is bounded and the arrangement of $\{\gamma=\gamma_0, ~\gamma_1,~\gamma_2,~...,~\gamma_{p-1}\}$ in the plane is important for the purpose of this article.\\
Define $$\mathcal{K}(\gamma)=\widehat{\C} \setminus \bigcup_{j=0}^{p-1}
\gamma_j.$$
The unbounded component of $\mathcal{K}(\gamma)$ is $n-$connected if and only if there are $n$ $H-$maximal nests.  
Note that $\mathcal{K}(\gamma)$ has exactly one unbounded component and
  at least one simply connected bounded component. Further, each of its components  intersects the Julia set.
 It is important to note that all the above observations are true for every $f^p$-invariant Jordan curve contained in a $p$-periodic Herman ring. Therefore, $\mathcal{K}$ is used instead of $\mathcal{K}(H)$ whenever the Herman ring containing $\gamma$ is understood.
 \par
  Given a Herman ring $H$, we first make following definitions.
\begin{definition}(\textbf{$H$-relevant pole})\\
A pole $w$ is called $H$-relevant if some $H_j$ surrounds $w$.
\end{definition}
$H$-relevant poles are always single separated by Remark~\ref{HR
encloses a sb pole}.
\begin{definition}(\textbf{Innermost ring}) A ring $H_i$ in an $H-$maximal nest $N$ is called innermost if it does not surround $H_j$ for any $j \neq i$.
\end{definition}
Each $H-$maximal nest has at least one innermost ring and exactly one outermost ring. Further, each $H_i$ belongs to exactly one $H-$ maximal nest.
\par
We are now to relate $H-$relevant poles and $H-$maximal nests.
\begin{lemma}
Let $H$ be a Herman ring of a function $f \in M_o$. Then the number of $H-$relevant poles is less than or equal to the number of $H-$maximal nests.
\label{relevantpoles<maximalnests}
\end{lemma}
\begin{proof}
No two $H-$relevant poles can be surrounded by a Herman ring by Remark~\ref{HR encloses a sb pole}. The proof follows from the fact that each $H-$relevant pole is in a single $H-$maximal nest.
\end{proof}
Every component of the pre-image of sufficiently small
neighbourhood of an omitted value is unbounded and on each of these
components, the map is not one-one. Now, it follows from the
definition (of Herman ring) that a Herman ring does not contain
any omitted value. Each periodic cycle of Herman rings of a function belonging to $M_o$
 contains a ring  surrounding the set of all omitted values as well as a ring surrounding a pole by Lemma~\ref{general1}. These two rings may be the same.
 Since invariant Herman rings do not exist for these functions
~\cite{tkzheng-omitted}, locating the  omitted
values and the poles in different components of $\mathcal{K}$ for a Herman ring $H$ requires some work.
\par
 Given a $p$-periodic Herman ring $H=H_0$ and a non-empty set $A$ contained in a bounded component of $\mathcal{K}$  such that $A \bigcap (\bigcup_{i=0}^{p-1}H_i) =\emptyset$, we say $H_j$ is the innermost ring with respect to $A$ if there is no $H_k, k \neq j$ surrounding $A$ and surrounded by $H_j$.

\begin{lemma} Let $H$ be a $p$-periodic Herman ring of  $f \in M_o$. Then we have the following.
\begin{enumerate}
\item The outermost ring of each  $H$-maximal nest surrounds at most one
pole.
\item The set $O_f$ is contained in a single bounded component of
$\mathcal{K}$.
\item Either the innermost ring $H(in)$ with respect to $O_f$
does not surround any pole or is an innermost ring of the
$H$-maximal nest containing it.
\item If only one $H$-maximal nest $N$ has
more than one ring in it then an innermost
 ring of $N$ surrounds $O_f$.
 \end{enumerate}
\label{O_f in innermost}\end{lemma}
\begin{proof}
\begin{enumerate}
\item
That the outermost ring of each  $H$-maximal nest surrounds at most one pole follows from Remark~\ref{HR encloses a sb
pole}.
\item   We assume $|O_f|=2$ otherwise there
is nothing to prove. Suppose that  $O_f$ intersects two distinct components of $\mathcal{K}$. Then
   there is a Herman ring, say $H_j$ surrounding exactly one omitted value of $f$.
If $H_{j-1}$ is the  Herman ring such that $f(H_{j-1})=H_j$ then consider an $f^p$-invariant Jordan curve $\gamma$ in $H_{j-1}$. It now follows from Remark~\ref{sb pole enclosed} that $f$ is not analytic on $B(\gamma)$ and hence not on $B(H_{j-1})$.
Therefore $B(H_{j-1})$ contains a pole  and hence $H_j$  surrounds
both the omitted values by Lemma~\ref{general1}. This contradicts
our assumption proving that $O_f$ is contained in a single component
of $\mathcal{K}$.
\item
If the innermost ring $H(in)$ with respect to $O_f$ surrounds a pole then $f^n(H(in))$
surrounds or is equal to $H(in)$ for all $n$ by Lemma~\ref{general1}.
That means $H(in)$ is the innermost ring of the $H$-maximal nest
containing it.

\item By Lemma~\ref{even}, there are at least two distinct rings $H_i$ and $H_j$ each surrounding some pole of $f$. It follows from Lemma~\ref{general1} that  $H_{i+1}$ and $H_{j+1}$ are distinct and each surrounds the set $O_f$, of all omitted values of $f$. Since $N$ is the only $H$-maximal nest having
more than one ring, these rings  $H_{i+1}$ and $H_{j+1}$ must be in $N$. In particular, there is a ring in $N$ surrounding $O_f$.
\par
 If $O_f$ is not surrounded by an innermost ring of
$N$ then the innermost ring with respect to $O_f$, say $H_k$ is different from all the innermost rings of $N$. Further $H_k$ surrounds a ring $H_{k'}$ of $N$. By $(3)$ of this lemma, $H_k$ surrounds a pole.
But in this situation, $f^n(H_k)$ is different from $H_{k'}$ for all $n$ by Lemma~\ref{general1} which is a contradiction.
Thus, all the omitted values are surrounded by an
innermost ring of $N$.
\end{enumerate}\end{proof}
%

\section{Proofs of the Results}

 \subsection{ Arrangement of Herman rings}
\begin{proof}[Proof of Theorem~\ref{nested-nonnested HR}]
Let  $H_0$ be a $p$-periodic Herman ring.
A function belonging to $M_o$ has no invariant Herman ring.
The proof appeared in ~\cite{tkzheng-omitted}. But we present it here for the sake of completeness. Let $H$ be an invariant Herman ring of $f$
 Taking a Jordan curve $\gamma$ in $H$ such that $f(\gamma)=\gamma$, it is seen that $B(\gamma)$ contains a pole and all the omitted values by Lemma~\ref{general1}. Since $f(B(\gamma))$ contains a neighborhood of $\infty$, we have $f(B(\gamma))=\widehat{\mathbb{C}} \setminus B(\gamma)$. However, this is not possible as $B(\gamma) \bigcap H$ contains $f-$invariant Jordan curves.
  \par
  We assume $p>1$ and then the proof will follow by deriving contradictions in each of the following cases.\\
\textbf{Case I: $H$ is nested}\\
 Without loss of generality suppose that  $H_i \subset
B(H_0)$ for all $i > 0$. Then by Lemma~\ref{general1}, there is  a pole of $f$ in $B(H_0)$ and all the omitted values are in $B(H_1)$.
 Further, the set $f(B(H_0))$ is the unbounded component of $\hC \setminus H_1$ by Lemma~\ref{one-one}. Since $\bigcup_{i=1}^{p-1} H_i \subset B(H_0)$,
 we must have $\bigcup_{i=2}^{p-1} H_i \bigcup H_0 \subset f(B(H_0))=\widehat{\C} \setminus (H_1 \bigcup B(H_1))$. Then $H_1$ is inner in the sense that $B(H_1)$ does not contain  $H_{j}$ for any $ j \geq 0$.
If $B(H_1)$ contains a component of $f^{-1}(H_j)$ for some $j$ then this component is not periodic and is in $B(H_0)$ (since $H_1 \subset B(H_0)$).
Also, there is a periodic Herman
ring mapped onto $H_j$ by $f$ which is either $H_0$ or in $B(H_0)$. Both these periodic and non-periodic components are mapped
onto $H_j$ and are in $B(H_0)$ giving that $f~:~{B(H_0)} \to \hC$ is
not one-one: a contradiction to Lemma~\ref{one-one}. Thus $B(H_1)$
does not contain any component of $f^{-1}(H_j)$ for any $j$ and therefore
$f(B(H_1))$ contains no $H_j$ for any $~j \geq 0$. Now, if $f$ has a
pole in $B(H_1)$ then $H_2 = H_0$. Otherwise, that means if $f$ is
analytic in $B(H_1)$ then $H_2$ is inner. Repeating this argument
and noting that $B(H_1)$ intersects the Julia set and hence,
contains pre-images of poles,
 the smallest natural number $j^*$ can be found such that $H_j$ is inner for all $j, 0< j< j^*$, $H_{j*-1}$ surrounds a pole and  $H_{j^*}=H_0$.
 This means that the forward image of every inner ring is inner or $H_0$ and that of the $H_0$ is inner. Now, take an $f^p$-invariant Jordan curve $\gamma_0$ in $H_0$
  and consider the region $A$ bounded by $\gamma=\bigcup_{i=0}^{j^*-1} f^{i}(\gamma_0)$. Any pole in $A$ would violate the univalence
  of $f$ in $B(H_0)$, since there is a pole in $B(H_{j^* -1}) \subset B(H_0)$. Thus $f$ is conformal in $A$ by Lemma~\ref{one-one}.
  Further $f(A)=A$ because $f(\partial A)=f(\gamma)=\gamma$ which gives that $f^n(A)=A$ for all $n$.
   But this is not possible as $A$ intersects the Julia set.\\
\textbf{Case II: $H$ is strictly non-nested}\\
 There are at least two rings, say $H_i$ and $H_j$ surrounding some pole of $f$ by Lemma~\ref{even}. Consequently, there are two Herman rings $H_{i+1}$ and $H_{j+1}$ surrounding $O_f$. But this is not possible as $H_0$ is strictly non-nested.
\par
 If there is a Herman ring of period two then it is nested or strictly non-nested.
Similarly, if a pole of $f$ is an omitted value then the forward
images of each ring, surrounding this pole, also surrounds the
pole. That means the Herman ring is nested. But these kinds of
Herman rings are not possible. This completes the proof.
\end{proof}
\subsection{Functions with at most two poles}
In order to prove Theorem~\ref{onepole-no HR}, we first prove a lemma.
\par
As discussed in the proof of Theorem~\ref{nested-nonnested HR}, a Herman  ring may have many pre-images out of which only one will be
periodic and we refer it as periodic pre-image of the ring.
For a Herman ring $H$ and an $H-$relevant pole $w$, let $H_j$ be a Herman ring surrounding $w$.
Then the $H-$outermost ring of the $H-$maximal nest containing $H_j$ is called the outermost ring with respect to $w$.

\begin{lemma}
Let $f \in M_o$  have only one single separated  pole $w$. If $H_m$ is
the outermost ring with respect to $w$ then $B(H_m) \bigcap O_f =
\emptyset$ and $H_m$ does not surround any pre-image of $w$.
\label{maximalring}
\end{lemma}
\begin{proof}
If $B(H_m) \bigcap O_f \neq \emptyset$ then the periodic pre-image
$H_{m-1}$ of $H_m$ must surround  a pole by Remark~\ref{sb pole
enclosed}. This pole is $w$ as this is the only single separated pole of $f$.
 Note that $O_f$ is surrounded by the innermost ring with
respect to $w$ by Lemma~\ref{O_f in innermost}(3). Since the periodic
pre-image $H_{m-1}$ of $H_m$ surrounds $w$, $H_{m-1}$ also surrounds
$O_f$. Similarly, it follows that the periodic pre-image $H_{m-2}$
of $H_{m-1}$ surrounds $w$ as well as $O_f$.
 By repeating this process for finitely many times, it is seen that $H_{j}$ surrounds $w$ for all $j \geq 0$.
  In other words, $H_m$ is strictly nested which is not possible by Theorem~\ref{nested-nonnested HR}.
\par
If $H_m$ surrounds a pre-image $w_{-1}$ of $w$ then $H_{m+1}$
surrounds $w$ as well as $O_f$ and therefore $B(H_m) \bigcap O_f \neq
\emptyset$. This is already shown to be impossible.
\end{proof}
\begin{remark} The lemma above implies that both the forward image and the
periodic pre-image of any $H_{m}$ are in nests different from the one
surrounding $w$.
\end{remark}
%
\begin{proof}[Proof of Theorem~\ref{onepole-no HR}]
Suppose that $H$ is a $p$-periodic Herman ring of $f$.
 Then some $H_i$ surrounds a pole, say $w$ by Lemma~\ref{general1}. Let all such Herman rings be enumerated as $\{H_{i_k} \}_{k=1}^{n}$. Also, let $H_{i_1}$ be the innermost and $H_{i_n}$ be the outermost  ring with respect to $w$. Note that each $H_{i_k +1}$ surrounds $O_f$ by Lemma~\ref{general1}. By Lemma~\ref{maximalring}, none of the omitted values lie in $B(H_{i_n})$. Therefore, each of  $\{H_{i_k +1} \}_{k=1}^{n}$ is in an $H-$maximal nest not containing any $H_{i_k}, k=1,~2,~...,~n$. Let $m$ be the smallest natural number such that $H_{{i_n}+1+m}$, the Herman ring containing $f^{m}(H_{i_{n} +1})$ surrounds a pole. This pole is none other than $w$ by our assumption. Consequently
$H_{{i_n}+1+m}=H_{i_j}$ for some $j$. This $j$ can only be $1$, otherwise there will be at least $n+1$ Herman rings surrounding $w$ which is not true. Thus $H_{{i_n}+1+m}=H_{i_1}$ and consequently
$H_{{i_1}+1+m} =H_{{i_n}}$. Note that $p=2(m+1)$. Consider a non-contractible $f^{p}$-invariant Jordan curve
$\gamma_1$ in $H_{i_1}$. Then $\gamma_{m+1}=f^{m+1}(\gamma_1)$ is also a non-contractible Jordan curve in $H_{i_n}$ and $f^{m+1}(\gamma_{m+1})=\gamma_1$. The map $f^{m+1}$ takes the region bounded by
these two curves $\gamma_{m+1}$ and $\gamma_{1}$ conformally onto itself. This is not possible if
$n >1$ as this region intersects the Julia set. But by Lemma~\ref{even}, $n=1$ is not possible. This contradiction completes the proof.
\end{proof}
\begin{remark}
If a function belonging to $M_o$ has  more than one pole but only one single separated pole then it does not have any Herman ring.
\label{2sb-poles}\end{remark}
Now we present the proof of Theorem~\ref{oddperiodic}.
\begin{proof}[Proof of Theorem~\ref{oddperiodic}]
If there are only two $H$-maximal nests and  one of them consists of
only one ring then the other, we call it bigger, surrounds $O_f$ in
one of its innermost ring by Lemma~\ref{O_f in innermost}(4).
There are at most two $H-$relevant poles by Lemma~\ref{relevantpoles<maximalnests}.
It follows from Lemma~\ref{even} that the number of
$H-$relevant poles is at least two. Hence, there are exactly two $H-$relevant poles.
However, there cannot be two poles surrounded by a single Herman ring by Lemma~\ref{one-one}.
Therefore, there is a pole $w$ surrounded by the outermost ring of the bigger nest.
Suppose that this pole and $O_f$ are separated by at least two rings, say $H^1$ and $H^2$. None of these rings surrounds any pole which gives that neither $f(H^1)$ nor $f(H^2)$ surrounds $O_f$. Since each ring in the bigger nest surrounds $O_f$ it follows that $f(H^1)$ and $f(H^2)$ are in an $H-$maximal nest which is different from the bigger nest and has more than one ring. However this is not possible.
 Therefore, the pole $w$ and $O_f$ are separated by only one ring of
the bigger nest. All rings of the bigger nest except the innermost surrounds $w$ since the bigger nest is strictly nested. The single ring of the other nest also surrounds a pole making the total number of
rings surrounding some pole even. Thus there are odd number of rings in the cycle containing $H$ proving that $H$ is odd periodic.

\end{proof}

 \subsection{Doubly connected Fatou components}
    For proving Theorem~\ref{open}, we put together Theorems 1-5  of \cite{tkzheng-omitted} as a
    lemma. Let $c(U)$ denote the connectivity of a domain $U$. We say a Fatou
component $V$ is  SCH if one of the following holds.
\begin{enumerate}
 \item $V$ is simply connected.
\item $V$ is multiply connected with $c(V_n)>1$ for all $n \in \mathbb{N}$ and
 $V_{\bar{n}}$ is a Herman ring for some $\bar{n} \in
\mathbb{N} \bigcup \{0\}$.
\end{enumerate}
A singleton Julia component is called buried if it is not in the boundary of any Fatou component.
Let $M_o ^k=\{f \in M_o~:~f ~\mbox{has}~ k ~\mbox{omitted values}\}$
for $k=1,~2$.
\begin{lemma}
   Let $f \in M_o$.
   \begin{enumerate}
   \item Let  $\mathcal{J}(f) \bigcap O_f \neq \emptyset$. If
$f \in M_o^2$ or $f \in M_o^1$ with $ |\mathcal{J}_{O_f}|>1$, then
each Fatou component of $f$ is SCH.
   \item  Let the set $O_f$ intersect two distinct Fatou components
$U_1$ and $U_2$ of $f$. If both  $U_1$ and $U_2$ are unbounded, or
exactly one of them is unbounded and is simply connected then all
the Fatou components of $f$ are simply connected. Otherwise, each
Fatou component of $f$ is SCH.
\item Let $O_f$ be contained in a Fatou component $ U$ and
$V$ be a Fatou component with $V_n \neq U$ for any $n
   \geq 0$.
  \begin{enumerate}
  \item If $U$ is unbounded, then $c(V_n)=1$ for all $n
   \geq 0$.
  \item If $U$ is bounded, then $V$ is SCH.
  \item If $U$ is wandering, then $c(U_n)=1$ for all $n
   \geq 0$.
\item Let $U$ be pre-periodic but not periodic.
If $U$ is unbounded, then $c(U_n)=1$ for all $n
   \geq 0$. If $U$ is bounded, then $U$ is SCH.
\item If $U$ is periodic, then $c(U_n)=1$ or $\infty$ for all $n
   \geq 0$.

\end{enumerate}
   \item Let $f \in M_o^1$, $O_f=\{a\} \subset
\mathcal{J}(f)$
 and $|\mathcal{J}_{a}|=1$. If $\mathcal{J}_{a}$ is not a buried component of
the Julia set, then $f$ has an infinitely connected Baker domain $B$
with period $p >1$ and $a$ is a pre-pole. Further, for each multiply
connected Fatou component $U$ of $f$ not landing on any Herman ring,
there is a non-negative integer $n$ depending on $U$ such that
$U_n=B$. In this case, singleton buried components are dense in
$\mathcal{J}(f)$.
\item
Let $f \in M_o^1$, $O_f=\{a\} \subset \mathcal{J}(f)$ and
$|\mathcal{J}_{a}|=1$. If $\mathcal{J}_{a}$ is a buried component of
the Julia set, then all the multiply connected Fatou components not
landing on any Herman ring are wandering and $a$ is a limit point of
  $\{f^n\}_{n>0}$ on each of these wandering domains. Further, if $\mathcal{F}(f)$ has a  multiply connected
   wandering domain, then the forward orbit of $a$ is an infinite set and singleton buried components are dense
    in $\mathcal{J}(f)$.
   \end{enumerate}
\label{MultconnFC}

   \end{lemma}
The above result describes the forward orbits of all multiply connected Fatou components in every possible situation. The proof basically depends on Lemma~\ref{general1}. The main idea of the proof is to analyse the forward orbits of non-contractible Jordan curves in multiply connected Fatou components.
%
\begin{proof}[Proof of Theorem~\ref{open}]
Let $D$ be a doubly connected periodic Fatou component of $f \in M_o$.
\par
 Let
$\mathcal{J}(f) \bigcap O_f \neq \emptyset$. If (1)$f \in M_o^2$,
(2) $f \in M_o^1$ with $ |\mathcal{J}_{O_f}|>1$, or (3) $f \in
M_o^1$, $O_f=\{a\} \subset \mathcal{J}(f)$, $|\mathcal{J}_{a}|=1$
and $\mathcal{J}_{a}$ is a buried component of the Julia set then
$D$ is a Herman ring by
Lemma~\ref{MultconnFC}(1) and (5). If $f \in M_o^1$, $O_f=\{a\}
\subset \mathcal{J}(f)$, $|\mathcal{J}_{a}|=1$ and $\mathcal{J}_{a}$
is not  a buried component of the Julia set then the Baker domain
(as mentioned in Lemma~\ref{MultconnFC}(4)) is infinitely connected.
Indeed, it can be shown  (see Lemma~4 and the proof of Theorem ~4 of \cite{tkzheng-omitted}) that all its forward images are also
infinitely connected. Thus in this
situation also, $D$ is a
Herman ring.
\par
Let $\mathcal{J}(f) \bigcap O_f = \emptyset$. If $O_f$ intersects
two Fatou components then we are done by Lemma~\ref{MultconnFC}(2).
Let $O_f \subset U$ for some Fatou component. Then, by
Lemma~\ref{MultconnFC}(3)(a-b), $D$ is a Herman ring whenever $D_n
\neq U$. If $D_n = U$ for some $n$ then $U$ must be periodic and by
Lemma~\ref{MultconnFC}(3)(e), $c(D_n)=1$ or $\infty$ for all $n$, which can not be possible as $D$ is periodic and doubly connected.
\end{proof}

\bibliographystyle{amsplain}

\begin{thebibliography}{10}
%
%
%
%

\bibitem{bk4}
I.~N. Baker, J.~Kotus, and L.~Yinian, \textit{Iterates of meromorphic functions
  {IV}: Critically finite functions}, Results Math. \textbf{22}
  (1992), 651--656.

\bibitem{bkdom-98}
I.N. Baker and P.~Dominguez, \textit{Analytic self-maps of the punctured plane},
 Complex variables Theory Appl. \textbf{37} (1998), no.~1-4, 67--91.

\bibitem{bergere95}
W.~Bergweiler and A.~E. Eremenko, \textit{On the singularities of the inverse to
  a meromorphic function of finite order}, Revista Math. Iberoamericana
  \textbf{11} (1995), 355--373.

\bibitem{bolsch99}
A.~Bolsch, \textit{Periodic {F}atou components of meromorphic functions}, Bull.
  London Math. Soc. \textbf{31} (1999), 543--555.
\bibitem{fagella2003-2}
P. Dom\'{i}nguez and N. Fagella, \textit{Existence of Herman rings for meromorphic functions},
Complex Variables Theory Appl. \textbf{49} (2004), no.~12, 851--870.

\bibitem{fagella-peter 2012}
N. Fagella and J. Peter,\textit{On the configuration of Herman rings of meromorphic functions},
J. Math. Anal. Appl. \textbf{394} (2012), 458--467.
\bibitem{tkzheng-omitted}
T.~Nayak and J.~H. Zheng, \textit{Omitted values and dynamics of transcendental
  meromorphic functions}, J. London Math. Soc. \textbf{83} (2011), no.~1,
  121--136.
\bibitem{gong-1996}
Gong Z.M, Qiu W.Y, and Ren F.Y, \textit{A negative answer to a problem of Bergweiler}, Complex Variables Theory Appl. \textbf{30} (1996), no.~4, 315--322.
\bibitem{zheng-2000}
J. Zheng, \textit{Remarks on Herman rings of transcendental meromorphic functions}, Indian J. Pure Appl. Math. \textbf{31 (7)} (2000), 747--751.
\end{thebibliography}

\end{document}